\documentclass{article}

\usepackage{amsfonts}
\usepackage{amssymb}
\usepackage{amsthm}
\usepackage{amscd}
\usepackage{enumerate}

\newtheorem{theorem}{Theorem}

\newtheorem{lemma}[theorem]{Lemma}
\newtheorem{corollary}[theorem]{Corollary}
\newtheorem{proposition}[theorem]{Proposition}

\newtheorem{problem}{Problem}

\title{Limit $T$-subspaces and the central polynomials in $n$ variables of the Grassmann algebra}

\author{
Dimas Jos\'e Gon\c{c}alves\thanks{e-mail \texttt{dimas@mat.unb.br}},
\\
Alexei Krasilnikov\thanks{e-mail \texttt{alexei@unb.br}},
\\
Irina Sviridova\thanks{e-mail \texttt{irina@mat.unb.br}}
\\
\\
Departamento de Matem\'atica,\\
Universidade de Bras\'\i lia,\\
70910-900 Bras\'\i lia, DF, Brazil }

\date{}

\begin{document}

\maketitle

\begin{abstract}
Let $F \langle X \rangle$ be the free unitary associative algebra
over a field $F$ on the set $X = \{ x_1,x_2, \ldots \}$. A vector
subspace $V$ of $F \langle X \rangle$  is called a
\emph{$T$-subspace} (or a \emph{$T$-space}) if $V$ is closed under
all endomorphisms of $F \langle X \rangle$. A $T$-subspace $V$ in
$F \langle X \rangle$ is \emph{limit} if every larger $T$-subspace
$W \gneqq V$ is finitely generated (as a $T$-subspace) but $V$
itself is not. Recently Brand\~ao Jr., Koshlukov, Krasilnikov and
Silva have proved that over an infinite field $F$ of
characteristic $p>2$ the $T$-subspace $C(G)$ of the central
polynomials of the infinite dimensional Grassmann algebra $G$ is a
limit $T$-subspace. They conjectured that this limit $T$-subspace
in $F \langle X \rangle$ is unique, that is, there are no limit
$T$-subspaces in $F \langle X \rangle$ other than $C(G)$. In the
present article we prove that this is not the case. We construct
infinitely many limit $T$-subspaces $R_k$ $(k \ge 1)$ in the
algebra $F \langle X \rangle$ over an infinite field $F$ of
characteristic $p>2$. For each $k \ge 1$, the limit $T$-subspace
$R_k$ arises from the central polynomials in $2k$ variables of the
Grassmann algebra $G$.
\end{abstract}



\section{Introduction}

Let $F$ be a field, $X$ a non-empty set and let $F \langle X
\rangle$ be the free unitary associative algebra over $F$ on the
set $X$. Recall that a \textit{T-ideal} of $F \langle X \rangle$
is an ideal closed under all endomorphisms of $F \langle X
\rangle$. Similarly, a \emph{$T$-subspace} (or a \emph{$T$-space})
is a vector subspace in $F \langle X \rangle$ closed under all
endomorphisms of $F \langle X \rangle$.

Let $I$ be a $T$-ideal in $F \langle X \rangle$. A subset $S
\subset I$ \textit{generates $I$ as a $T$-ideal} if $I$ is the
minimal $T$-ideal  in $F \langle X \rangle$ containing $S$. A
$T$-subspace of $F \langle X \rangle$ generated by $S$ (as a
$T$-subspace) is defined in a similar way. It is clear that the
$T$-ideal ($T$-subspace) generated by $S$ is the ideal (vector
subspace) generated by all the polynomials $f(g_1,\ldots, g_m)$,
where $f=f(x_1, \ldots , x_m) \in S$ and $g_i\in F \langle X
\rangle$ for all $i$.

Note that if $I$ is a $T$-ideal in $F \langle X \rangle$ then
$T$-ideals and $T$-subspaces can be defined in the quotient
algebra $F \langle X \rangle/I$ in a natural way. We refer to
\cite{drbook, drformanekbook, gz, K-BR, kemerbook, rowenbook} for
the terminology and basic results concerning $T$-ideals and
algebras with polynomial identities and to \cite{BOR_cp, BKKS,
GrishCyb, grshchi, K-BR} for an account of the results concerning
$T$-subspaces.

From now on we write $X$ for $\{ x_1, x_2, \ldots \}$ and $X_n$
for $\{ x_1, \ldots , x_n \}$, $X_n \subset X$. If $F$ is a field
of characteristic $0$ then every $T$-ideal in $F \langle X
\rangle$ is finitely generated (as a $T$-ideal); this is a
celebrated result of Kemer \cite {Kemer, kemerbook} that solves
the Specht problem. Moreover, over such a field $F$ each
$T$-subspace in $F \langle X \rangle$ is finitely generated; this
has been proved more recently by Shchigolev \cite{Shchigolev01}.
Very recently Belov \cite{Belov10} has proved that, for each
Noetherian commutative and associative unitary ring $K$ and each
$n \in \mathbb N$, each $T$-ideal in $K \langle X_n \rangle$ is
finitely generated.

On the other hand, over a field $F$ of characteristic $p>0$ there
are $T$-ideals in $F \langle X \rangle$ that are not finitely
generated. This has been proved by Belov \cite{Bel99}, Grishin
\cite{Grishin99} and Shchigolev \cite{Shchigolev99} (see also
\cite{Bel00, Grishin00, K-BR}). The construction of such
$T$-ideals uses the non-finitely generated $T$-subspaces in $F
\langle X \rangle$ constructed by Grishin \cite{Grishin99} for
$p=2$ and by Shchigolev \cite{Shchigolev00} for $p>2$ (see also
\cite{Grishin00}). Shchigolev \cite{Shchigolev00} also constructed
non-finitely generated $T$-subspaces in $F \langle X_n \rangle$,
where $n>1$ and $F$ is a field of characteristic $p>2$.

A $T$-subspace $V^*$ in $F \langle X \rangle$ is called
\emph{limit} if every larger $T$-subspace $W \gneqq V^*$ is
finitely generated as a $T$-subspace but $V^*$ itself is not. A
\textit{limit $T$-ideal} is defined in a similar way. It follows
easily from Zorn's lemma that if a $T$-subspace $V$ is not
finitely generated then it is contained in some limit $T$-subspace
$V^*$. Similarly, each non-finitely generated $T$-ideal is
contained in a limit $T$-ideal. In this sense limit $T$-subspaces
($T$-ideals) form a ``border'' between those $T$-subspaces
($T$-ideals) which are finitely generated and those which are not.

By \cite{Bel99, Grishin99, Shchigolev99}, over a field $F$ of
characteristic $p>0$ the algebra $F \langle X \rangle$ contains
non-finitely generated $T$-ideals; therefore, it contains at least
one limit $T$-ideal. No example of a limit $T$-ideal is known so
far. Even the cardinality of the set of limit $T$-ideals in $F
\langle X \rangle$ is unknown; it is possible that, for a given
field $F$ of characteristic $p>0$, there is only one limit
$T$-ideal. The non-finitely generated $T$-ideals constructed in
\cite{AladovaKras} come closer to being limit than any other known
non-finitely generated $T$-ideal. However, it is unlikely that
these $T$-ideals are limit.

About limit $T$-subspaces in $F \langle X \rangle$ we know more
than about limit $T$-ideals. Recently Brand\~ao Jr., Koshlukov,
Krasilnikov and Silva \cite{BKKS} have found the first example of
a limit $T$-subspace in $F \langle X \rangle$ over an infinite
field $F$ of characteristic $p>2$. To state their result precisely
we need some definitions.

For an associative algebra $A$, let $Z(A)$ denote the centre of
$A$,
\[
Z(A) = \{ z \in A \mid za= az \mbox{ for all } a \in A \}.
\]
A polynomial $f(x_1,\ldots,x_n)$ is \emph{a central polynomial}
for $A$ if $f(a_1,\ldots, a_n) \in Z(A)$ for all $a_1, \dots , a_n
\in A$. For a given algebra $A$, its central polynomials form a
$T$-subspace $C(A)$ in $F \langle X \rangle$. However, not every
$T$-subspace can be obtained as the $T$-subspace of the central
polynomials of some algebra.

Let $V$ be the vector space over a field $F$ of characteristic
$\ne 2$, with a countable infinite basis $e_1$, $e_2, \dots $ and
let $V_s$ denote the subspace of $V$ spanned by $e_1, \ldots , e_s$
$(s = 2,3 , \ldots ) .$ Let $G$ and $G_s$ denote the unitary
Grassmann algebras of $V$ and $V_s$, respectively. Then as a
vector space $G$ has a basis that consists of 1 and of all
monomials $e_{i_1}e_{i_2}\ldots e_{i_k}$, $i_1<i_2<\cdots<i_k$,
$k\ge 1$. The multiplication in $G$ is induced by $e_ie_j=-e_je_i$
for all $i$ and $j$. The algebra $G_s$ is the subalgebra of $G$
generated by $e_1, \ldots ,e_s$, and $ \mbox{dim }G_s = 2^s$. We
refer to $G$ and $G_s$ $(s = 2,3, \ldots )$ as to the infinite
dimensional Grassmann algebra and the finite dimensional Grassmann
algebras, respectively.

The result of \cite{BKKS} concerning a limit $T$-subspace is as
follows:

\begin{theorem}[\cite{BKKS}]
\label{C_limit} Let $F$ be an infinite field of characteristic
$p>2$ and let $G$ be the infinite dimensional Grassmann algebra
over $F$. Then the vector space $C(G)$ of the central polynomials
of the algebra $G$ is a limit T-space in $F \langle X \rangle$.
\end{theorem}

It was conjectured in \cite{BKKS} that a limit $T$-subspace in $F
\langle X \rangle$ is unique, that is, $C(G)$ is the only limit
$T$-subspace in $F \langle X \rangle$. In the present article we
show that this is not the case. Our first main result is as
follows.

\begin{theorem}
\label{theorem_main1} Over an infinite field $F$ of characteristic
$p>2$ the algebra $F \langle X \rangle$ contains infinitely many
limit $T$-subspaces.
\end{theorem}

Let $F$ be an infinite field of characteristic $p >0.$ In order to
prove Theorem \ref{theorem_main1} and to find infinitely many
limit $T$-subspaces in $F \langle X \rangle$ we first find limit
$T$-subspaces in $F \langle X_n \rangle$ for $n = 2 k$, $k \ge 1$.
Let $C_n = C(G) \cap F \langle X_n \rangle$ be the set of the
central polynomials in at most $n$ variables of the unitary Grassmann
algebra $G$. Our second main result is as follows.

\begin{theorem}
\label{theorem_main2} Let $F$ be an infinite field of
characteristic $p>2.$ If $n = 2k$, $k \ge 1$, then $C_{n}$ is a
limit $T$-subspace in $F \langle X_{n} \rangle$. If $n = 2k+1$, $k
>1$, then $C_n$ is finitely generated as a $T$-subspace in $F
\langle X_n \rangle$.
\end{theorem}

\noindent \textbf{Remark. } We do not know whether the
$T$-subspace $C_3$ is finitely generated.

\bigskip

Define $[a,b]=ab-ba$, $[a,b,c] = [[a,b],c]$. For $k \ge 1$, let
$T^{(3,k)}$ denote the $T$-ideal in $F \langle X \rangle$
generated by $[x_1,x_2,x_3]$ and $[x_1,x_2][x_3,x_4] \ldots
[x_{2k-1},x_{2k}]$ and let $R_k$ denote the $T$-subspace in $F \langle X
\rangle$ generated by $C_{2k}$ and $T^{(3,k+1)}$. Theorem
\ref{theorem_main1} follows immediately from our third main result
that is as follows.

\begin{theorem}
\label{theorem_main3} Let $F$ be an infinite field of
characteristic $p>2.$ For each $k \ge 1$, $R_k$ is a limit
$T$-subspace in $F \langle X \rangle$. If $k \ne l$ then $R_k \ne
R_l$.
\end{theorem}

Now we modify the conjecture made in \cite{BKKS}.

\begin{problem}
Let $F$ be an infinite field of characteristic $p>2$. Is each
limit $T$-subspace in $F \langle X \rangle$ equal to either $C(G)$
or $R_k$ for some $k$? In other words, are $C(G)$ and $R_k$ $(k
\ge 1)$ the only limit $T$-subspaces in $F \langle X \rangle$?
\end{problem}

\medskip
In the proof of Theorems \ref{theorem_main2} and
\ref{theorem_main3} we will use the following theorem that has
been proved independently by Bekh-Ochir and Rankin \cite{BOR_cp},
by Brand\~ao Jr., Koshlukov, Krasilnikov and Silva \cite{BKKS} and
by Grishin \cite{Grishin10}. Let
\[
q(x_1,x_2)=x_1^{p-1}[x_1,x_2]x_2^{p-1}, \qquad
q_k(x_1,\dots,x_{2k})=q(x_1,x_{2}) \cdots q(x_{2k-1},x_{2k}).
\]

\begin{theorem}[\cite{BOR_cp}, \cite{BKKS}, \cite{Grishin10}]
\label{generators_of_C} Over an infinite field $F$ of a
characteristic $p>2$ the vector space $C(G)$ of the central
polynomials of $G$ is generated (as a T-space in $F \langle X
\rangle $) by the polynomial
\[
x_1[x_2,x_3,x_4]
\]
and the polynomials
\[
x_1^p \ , \ x_1^p \, q_1(x_2,x_3) \ , \ x_1^p \, q_2(x_2,x_3,x_4,
x_5) \ , \ldots ,\ x_1^p \, q_n(x_2, \ldots , x_{2n+1}) \, ,
\ldots .
\]
\end{theorem}

In order to prove Theorems \ref{theorem_main2} and
\ref{theorem_main3} we need some auxiliary results. Define, for
each $l \ge 0$,
\[
q^{(l)}(x_1,x_2)= x_1^{p^l-1}[x_1,x_2]x_2^{p^l-1},
\]
\[
q_{k}^{(l)}(x_1,\dots,x_{2k})= q^{(l)}(x_1,x_{2}) \cdots
q^{(l)}(x_{2k-1},x_{2k}).
\]
Recall that $C_n = C(G) \cap F \langle X_n \rangle$. To prove
Theorem \ref{theorem_main2} we need the following assertions that
are also of independent interest.

\begin{proposition}
\label{generators_of_C_n} If $n=2k$, $k>1$, then $C_n$ is
generated as a $T$-subspace in $F \langle X_n \rangle$ by the
polynomials
\[
x_1[x_2,x_3,x_4], \quad x_1^p, \quad x_1^p q_1(x_2,x_3), \quad
\ldots , \quad x_1^p q_{k-1}(x_2, \ldots , x_{2k-1})
\]
together with the polynomials
\[
\{ q_k^{(l)}(x_1, \ldots , x_{2k}) \mid l=1,2, \ldots \} .
\]
If $n=2k+1$, $k>1$, then $C_n$ is generated as a $T$-subspace in
$F \langle X_n \rangle$ by the polynomials
\[
x_1[x_2,x_3,x_4], \quad x_1^p, \quad x_1^p q_1(x_2,x_3), \quad
\ldots , \quad x_1^p q_{k}(x_2, \ldots , x_{2k+1}).
\]
\end{proposition}

Let $T^{(3)}$ denote the $T$-ideal in $F \langle X \rangle$
generated by $[x_1,x_2,x_3].$ Define $T^{(3)}_n = T^{(3)} \cap F
\langle X_n \rangle$. We deduce Proposition
\ref{generators_of_C_n} from the following.

\begin{proposition}
\label{generators_of_C_n/T3n} If $n=2k$, $k \ge 1$, then
$C_n/T^{(3)}_n$ is generated as a $T$-subspace in $F \langle X_n
\rangle /T^{(3)}_n$ by the polynomials
\begin{equation}\label{gen_C_n/T3n_1}
x_1^p + T^{(3)}_n, \quad x_1^p q_1(x_2,x_3)+T^{(3)}_n, \quad
\ldots , \quad x_1^p q_{k-1}(x_2, \ldots , x_{2k-1}) + T^{(3)}_n
\end{equation}
together with the polynomials
\begin{equation}\label{gen_C_n/T3n_2}
\{ q_k^{(l)}(x_1, \ldots , x_{2k})+ T^{(3)}_n \mid l=1,2, \ldots
\} .
\end{equation}
If $n=2k+1$, $k \ge 1$, then the $T$-subspace $C_n/T^{(3)}_n$ in
$F \langle X_n \rangle/T^{(3)}_n$ is generated by the polynomials
\begin{equation}\label{gen_C_n/T3n_3}
x_1^p + T^{(3)}_n, \quad x_1^p q_1(x_2,x_3)+T^{(3)}_n, \quad
\ldots , \quad x_1^p q_{k}(x_2, \ldots , x_{2k+1})+ T^{(3)}_n.
\end{equation}
\end{proposition}

\noindent \textbf{Remarks. } 1. For each $k \ge 1$, the limit
$T$-subspace $R_k$ does not coincide with the $T$-subspace $C(A)$
of all central polynomials of any algebra $A$.

Indeed, suppose that $R_k = C(A)$ for some $A$. Let $T(A)$ be the
$T$-ideal of all polynomial identities of $A$. Then, for each $f
\in C(A)$ and each $g \in F \langle X \rangle$, we have $[f,g] \in
T(A)$. Since $[x_1,x_2] \in R_k = C(A)$, we have $[x_1,x_2,x_3]
\in T(A)$. It follows that $T^{(3)} \subseteq T(A)$.

It is well-known that if a $T$-ideal $T$ in the free unitary
algebra $F \langle X \rangle$ over an infinite field $F$ contains
$T^{(3)}$ then either $T = T^{(3)}$ or $T = T^{(3, n)}$ for some
$n$ (see, for instance, \cite[Proof of Corollary 7]{gk}). Hence,
either $T(A) = T^{(3)}$ or $T(A) = T^{(3, n)}$ for some $n$. Note
that $T^{(3)} = T(G)$ and $T^{(3,n)} = T(G_{2n-1})$ (see, for
example, \cite{gk}) so we have either $T(A) = T(G)$ or $T(A) =
T(G_{2n-1})$ for some $n$.

For an associative algebra $B$, we have $f(x_1, \ldots , x_r) \in C(B)$ if and only if $[ f(x_1, \ldots , x_r), x_{r+1} ] \in T(B)$. It follows that if $B_1, B_2$ are algebras such that $T(B_1) = T(B_2)$ then $C(B_1) = C(B_2)$. In particular, if $T(A) = T(G)$ then $C(A) = C(G)$, and if $T(A) = T(G_{2n-1})$ then $C(A) = C(G_{2n-1})$.

However,
\[
x_1[x_2,x_3] \ldots [x_{2k+2},x_{2k+3}] \in R_k \setminus C(G)
\]
so $R_k \ne C(G)$. Furthermore, the $T$-subspaces $C(G_s)$ of the
central polynomials of the finite dimensional Grassmann algebras
$G_s$ $(s =2, 3, \ldots )$ have been described recently by
Bekh-Ochir and Rankin \cite{BOR_cpfd} and by Koshlukov,
Krasilnikov and Silva \cite{KKS}; these $T$-subspaces are finitely
generated and do not coincide with $R_k$. This contradiction
proves that $R_k \ne C(A)$ for any algebra $A$, as claimed.

2. For an associative unitary algebra $A$, let $C_n (A)$ and $T_n(A)$
denote the set of the central polynomials and the set of the
polynomial identities in $n$ variables $x_1, \ldots , x_n$ of $A$,
respectively; that is, $C_n (A) = C(A) \cap F \langle X_n \rangle$
and $T_n (A) = T(A) \cap F \langle X_n \rangle$. Then $C_n(A)$ is
a $T$-subspace and $T_n(A)$ is a $T$-ideal in $F \langle X_n \rangle$.

Note that, by Belov's result \cite{Belov10}, the $T$-ideal $T_n(A)$ is
finitely generated for each algebra $A$ over a Noetherian ring and
each positive integer $n$.
On the other hand, there exist unitary algebras $A$ over
an infinite field $F$ of characteristic $p>2$ such that, for some
$n>1$, the $T$-subspace $C_n (A)$ of the central polynomials of
$A$ in $n$ variables is not finitely generated. Moreover, such
an algebra $A$ can be finite dimensional. Indeed, take $A = G_s$,
where $s \ge n$. It can be checked that $C(G_s) \cap F \langle X_n
\rangle = C_n$ if $s \ge n$. By Proposition \ref{C2k_0}, the
$T$-subspace $C_{2k} (G_s)$ in $F \langle X_{2k} \rangle$ is
not finitely generated provided that $s \ge 2k$.

However, the following problem remains open.

\begin{problem}
Does there exist a finite dimensional algebra $A$ over an infinite
field $F$ of characteristic $p >0$ such that the $T$-subspace
$C(A)$ of all central polynomials of $A$ in $F \langle X \rangle$
is not finitely generated?
\end{problem}
Note that a similar problem for the $T$-ideal $T(A)$ of all
polynomial identities of a finite
dimensional algebra $A$ over an infinite field $F$ of characteristic
$p >0$ remains open as well; it is one of the most interesting and
long-standing open problems in the area.

\section{Preliminaries}

Let $\langle S \rangle^{TS}$ denote the $T$-subspace generated by
a set $S \subseteq F \langle X \rangle$. Then $\langle S
\rangle^{TS}$ is the span of all polynomials $f(g_1,\ldots, g_n)$,
where $f\in S$ and $g_i\in F \langle X \rangle$ for all $i$. It is
clear that for any polynomials $f_1,$ \dots, $f_s \in F \langle X
\rangle$ we have $\langle f_1, \dots, f_s \rangle^{TS}= \langle
f_1 \rangle^{TS}+\dots+\langle f_s \rangle^{TS}.$

Recall that a polynomial $f(x_1, \ldots , x_n) \in F \langle X
\rangle$ is called a \textit{polynomial identity} in an algebra
$A$ over $F$ if $f (a_1, \ldots , a_n) =0$ for all $a_1, \ldots ,
a_n \in A$. For a given algebra $A$, its polynomial identities
form a $T$-ideal $T(A)$ in $F \langle X \rangle$ and for every
$T$-ideal $I$ in $F \langle X \rangle$ there is an algebra $A$
such that $I = T(A)$, that is, $I$ is the ideal of all polynomial
identities satisfied in $A.$ Note that a polynomial
$f=f(x_1,\ldots,x_n)$ is central for an algebra $A$ if and only if
$[f,x_{n+1}]$ is a polynomial identity of $A.$

Let $f = f(x_1, \ldots , x_n) \in F \langle X \rangle$. Then $f =
\sum_{0 \le i_1, \ldots , i_n} f_{i_1 \ldots i_n},$ where each
polynomial $f_{i_1 \ldots i_n}$ is multihomogeneous of degree
$i_s$ in $x_s$ $(s = 1, \ldots , n).$ We refer to the polynomials
$f_{i_1 \ldots i_n}$ as to the \textit{multihomogeneous
components} of the polynomial $f.$ Note that if $F$ is an infinite
field, $V$ is a $T$-ideal in $F \langle X \rangle$ and $f \in V$
then $f_{i_1 \ldots i_n} \in V$ for all $i_1, \ldots , i_n$ (see,
for instance, \cite{Baht, drbook, gz, rowenbook}). Similarly, if
$V$ is a $T$-subspace in $F \langle X \rangle$ and $f \in V$ then
all the multihomogeneous components $f_{i_1 \ldots i_n}$of $f$
belong to $V$.

Over an infinite field $F$ the $T$-ideal $T(G)$ of the polynomial
identities of the infinite dimensional unitary Grassmann algebra
$G$ coincides with $T^{(3)}$. This was proved by Krakowski and
Regev \cite{krreg} if $F$ is of characteristic $0$ (see also
\cite{lat}) and by several authors in the general case, see for
example \cite{gk}.

It is well known (see, for example, \cite{krreg,
lat}) that over any field $F$ we have
\begin{eqnarray} \label{prop}
&&[g_1,g_2][g_1,g_3]+T^{(3)}=T^{(3)}; \nonumber\\
&&[g_1,g_2][g_3,g_4]+T^{(3)}=-[g_3,g_2][g_1,g_4]+T^{(3)}; \\
&&[g_1^m,g_2]+T^{(3)}=m g_1^{m-1}[g_1,g_2]+T^{(3)} \nonumber
\end{eqnarray}
for all $g_1,g_2,g_3,g_4 \in F \langle X \rangle$. Also it is well known (see, for
instance, \cite{BKKS, grshchi}) that a basis of the vector space $F\langle X
\rangle/T^{(3)}$ over $F$ is formed by the elements of the form
\begin{eqnarray} \label{basa}
x_{i_1}^{m_1} \cdots x_{i_d}^{m_d} [x_{j_1},x_{j_2}] \cdots
[x_{j_{2s-1}},x_{j_{2s}}]+T^{(3)},
\end{eqnarray}
where $d,s \ge 0$, $i_1< \dots <i_d,$ $j_1<\dots<j_{2s}.$

Define $T_n^{(3)} = T^{(3)} \cap F \langle X_n \rangle.$ We claim
that if $n < 2i$ then
\begin{equation} \label{T3iTn3}
T^{(3,i)} \cap F \langle X_n \rangle = T_n^{(3)}.
\end{equation}
Indeed, a basis of the vector space $( F
\langle X_n \rangle + T^{(3)})/T^{(3)}$ is formed by the elements
of the form (\ref{basa}) such that $1 \le i_1 < \ldots < i_d \le
n,$ $1 \le j_1< \ldots < j_{2s} \le n.$ In particular, we have $2s
\le n.$ On the other hand, it can be easily checked that
$T^{(3,i)}/T^{(3)}$ is contained in the linear span of the
elements of the form (\ref{basa}) such that $s \ge i$. Since $n <
2i$, we have
\[
((F \langle X_n \rangle  + T^{(3)})/T^{(3)}) \cap (T^{(3,i)}/T^{(3)}) = \{ 0 \} ,
\]
that is, $T^{(3,i)} \cap  F \langle X_n \rangle \subseteq
T^{(3)}$. It follows immediately that $T^{(3,i)} \cap  F \langle
X_n \rangle \subseteq T_n^{(3)}$. Since $T_n^{(3)}
\subseteq T^{(3,i)} \cap  F \langle X_n \rangle$ for all $i$,
we have $T^{(3,i)} \cap  F \langle X_n \rangle = T_n^{(3)}$
if $n < 2i$, as claimed.

Let $F$ be a field of characteristic $p>2.$ It is well known (see,
for example, \cite{regevca, BOR_cp, BKKS, GrishCyb}) that, for each
$g, g_1, \ldots , g_n \in F \langle X \rangle$, we have
\begin{eqnarray} \label{prop2}
&& g^p + T^{(3)} \mbox{ is central in } F \langle X \rangle /
T^{(3)}; \nonumber \\
&&(g_1 \cdots g_n)^p+T^{(3)} = g_1^p \cdots g_n^p+T^{(3)};\\
&&(g_1 +\dots + g_n)^p+T^{(3)}= g_1^p +\dots + g_n^p+T^{(3)}.
\nonumber
\end{eqnarray}

Let $F$ be an infinite field of characteristic $p>2$. Let
$Q^{(k,l)}$ be the $T$-subspace in $F\langle X \rangle$ generated
by $q_{k}^{(l)}$ $(l \ge 0)$, $Q^{(k,l)}=\langle q_{k}^{(l)}(x_1,
\dots, x_{2k}) \rangle^{TS}$. Note that the multihomogeneous
component of the polynomial
\begin{eqnarray*}
&&q_k^{(l)}(1+x_1, \ldots , 1+x_{2k}) \\
&&= (1 + x_1)^{p^l-1}[x_1,x_2](1 + x_2)^{p^l-1} \ldots (1 +
x_{2k-1})^{p^l-1}[x_{2k-1},x_{2k}](1 + x_{2k})^{p^l-1}
\end{eqnarray*}
of degree $p^{l-1}$ in all the variables $x_1, \ldots , x_{2k}$ is
equal to
\[
\gamma \, q_k^{(l-1)}(x_1, \ldots , x_{2k}) = \gamma \,
x_1^{p^{l-1}-1}[x_1,x_2]x_2^{p^{l-1}-1} \ldots
x_{2k-1}^{p^{l-1}-1}[x_{2k-1},x_{2k}]x_{2k}^{p^{l-1}-1},
\]
where $\gamma = {p^{l}-1 \choose p^{l-1}-1}^{2k} \equiv 1
\pmod{p}.$ It follows that $q_k^{(l-1)} \in Q^{(k,l)}$ for all $l
>0$ so $Q^{(k,l-1)} \subseteq Q^{(k,l)}$. Hence, for each $l >0$ we have
\begin{equation}\label{sum_q}
\sum \limits_{i=0}^{l} Q^{(k,i)} = Q^{(k,l)}.
\end{equation}

The following lemma is a reformulation of a result of Grishin and
Tsybulya \cite[Theorem 1.3, item 1)]{GrishCyb}.

\begin{lemma}\label{lemma_GTs}
Let $F$ be an infinite field of characteristic $p>2$. Let $k \ge
1$, $a_i \ge 1$ for all $i = 1, 2 \ldots , 2k$ and let
\[
m = x_1^{a_1-1}x_2^{a_2-1} \ldots x_{2k}^{a_{2k}-1}[x_1,x_2]\ldots
[x_{2k-1},x_{2k}] \in F \langle X \rangle.
\]
Suppose that, for some $i_0$, $1 \le i_0 \le 2k$, we have $a_{i_0}
= p^l b$, where $l \ge 0$ and $b$ is coprime to $p$. Suppose also
that, for each $i$, $1 \le i \le 2k$, we have $a_i \equiv 0
\pmod{p^l}$. Then
\[
\langle m \rangle^{TS} + T^{(3)} = Q^{(k,l)} + T^{(3)}.
\]
\end{lemma}

\section{Proof of Propositions \ref{generators_of_C_n} and \ref{generators_of_C_n/T3n}}

In the rest of the paper, $F$ will denote an infinite field of
characteristic $p>2$.

\subsection*{Proof of Proposition \ref{generators_of_C_n/T3n}}
Let $U$ be the $T$-subspace of $F \langle X_n \rangle$ defined as
follows:
\begin{itemize}
\item[i)] $T_n^{(3)} \subset U$;
\item[ii)] the $T$-subspace $U/T_n^{(3)}$ of $F \langle X_n
\rangle/T_n^{(3)}$ is generated by the polynomials
(\ref{gen_C_n/T3n_1}) and (\ref{gen_C_n/T3n_2}) if $n=2k$ and by
the polynomials (\ref{gen_C_n/T3n_3}) if $n=2k+1$.
\end{itemize}
To prove the proposition we have to show that
$C_n/T_n^{(3)}=U/T_n^{(3)}$ (equivalently, $C_n = U$). It can be
easily seen that $U/T_n^{(3)} \subseteq C_n/T_n^{(3)}$. Thus, it
remains to prove that $C_n/T_n^{(3)} \subseteq U/T_n^{(3)}$
(equivalently, $C_n \subseteq U$).

Let $h$ be an arbitrary element of $C_n$. We are going to check
that $h+T_n^{(3)} \in U/T_n^{(3)}$.

Since $h \in C(G)$, it follows from The\-o\-rem
\ref{generators_of_C} that
\[
h=\sum_{j} \alpha_{j} \  v_j^{p} + \sum_{i, j } \alpha_{i j} \
w_{i j}^{p} \  q_i(f_1^{(i j)}, \dots, f_{2i}^{(i j)})+ h',
\]
where $v_j, w_{i j}, f_s^{(i j)} \in F \langle X \rangle$,
$\alpha_j, \alpha_{i j} \in F$, $h' \in T^{(3)}$. Note that $h \in
F \langle X_n \rangle$ so we may assume that $v_j, w_{i j},
f_s^{(i j)}, h' \in F \langle X_n \rangle$ for all $i,j,s$. It
follows that
\[
h + T_n^{(3)} =\sum_{j} \alpha_{j} \  v_j^{p} + \sum_{i, j }
\alpha_{i j} \ w_{i j}^{p} \  q_i(f_1^{(i j)}, \dots, f_{2i}^{(i
j)})+ T_n^{(3)}.
\]

Recall that $T^{(3,i)}$ is the $T$-ideal in $F \langle X \rangle$
generated by the polynomials $[x_1,x_2,x_3]$ and $[x_1,x_2] \dots
[x_{2i-1},x_{2i}]$. By (\ref{T3iTn3}), we have $T^{(3,i)} \cap F
\langle X_n \rangle = T^{(3)}_{n}$ for each $i$ such that
$2i>n$. Since, for each $i, j$,
\[
w_{i j}^p \ q_i(f_1^{(i j)}, \dots, f_{2i}^{(i j)}) \in T^{(3,i)},
\]
we have
\[
 \sum_{i > \frac{n}{2}} \sum_{ j } \alpha_{i j} \ w_{i j}^{p}
 \  q_i(f_1^{(i j)}, \dots, f_{2i}^{(i j)}) \in T^{(3,i)} \cap F
\langle X_n \rangle = T_n^{(3)}.
\]
It follows that
\[
h + T_n^{(3)} =\sum_{j} \alpha_{j} \  v_j^{p} + \sum_{i \le
\frac{n}{2}} \sum_{ j } \alpha_{i j} \ w_{i j}^{p} \  q_i(f_1^{(i
j)}, \dots, f_{2i}^{(i j)})+ T_n^{(3)}.
\]

If $n=2k+1$ ($k \ge 1$) then we have
\[
h + T^{(3)}_{n} = \sum_{j} \alpha_{j} v_{j}^{p} + \sum_{i=1}^k
\sum_j \alpha_{i j} \ w_{i j}^{p} \ q_i(f_1^{(i j)}, \dots,
f_{2i}^{(i j)}) + T^{(3)}_{n}
\]
so $h + T^{(3)}_{n} \in U/T_n^{(3)}$, as required.

If $n=2k$ ($k \ge 1$) then we have
\[
h + T^{(3)}_{n} = h_1 + h_2 + T^{(3)}_{n},
\]
where
\[
h_1 = \sum_{j} \alpha_{j} v_j^{p} + \sum_{i=1}^{k-1} \sum_j
\alpha_{i j} \ w_{i j}^{p} \  q_i(f_1^{(i j)},\dots,f_{2i}^{(i
j)})
\]
and
\[
h_2 = \sum_j \alpha_{k j} \ w_{k j}^{p} \ q_k(f_1^{(k j)},\dots,
f_{2k}^{(k j)}).
\]
It is clear that $h_1 + T^{(3)}_{n}$ belongs to the $T$-subspace
generated by the polynomials (\ref{gen_C_n/T3n_1}); hence, $h_1 +
T^{(3)}_{n} \in U/T^{(3)}_{n}$. On the other hand, it can be
easily seen that $h_2 + T^{(3)}_{n}$ is a linear combination of
polynomials of the form $m + T^{(3)}_{n}$, where
\[
m = x_1^{b_1} \cdots x_{2k}^{b_{2k}}[x_1,x_2] \cdots [x_{2k-1},x_{2k}].
\]
We claim that, for each $m$ of this form, the polynomial $m + T_{2k}^{(3)}$ belongs to $U/T_{2k}^{(3)}$.

Indeed, by Lemma \ref{lemma_GTs}, we have $\langle m \rangle^{TS} + T^{(3)} = \langle q_k^{(l)} \rangle^{TS} + T^{(3)}$ for some $l \ge 0$. Since both $m$ and $q_k^{(l)}$ are polynomials in $x_1, \dots , x_{2k}$, this equality implies that $m + T_{2k}^{(3)}$ belongs to the $T$-subspace of $F \langle X_{2k} \rangle / T_{2k}^{(3)}$ that is generated by $q_k^{(l)} + T_{2k}^{(3)}$ for some $l \ge 0$. If $l \ge 1$ then $m + T_{2k}^{(3)}\in U/T_{2k}^{(3)}$ because, for $l \ge 1$, $q_k^{(l)} + T_{2k}^{(3)}$ is a polynomial of the form (\ref{gen_C_n/T3n_2}). If $l = 0$ then $m + T_{2k}^{(3)}$ belongs to the $T$-subspace of $F \langle X_{2k} \rangle / T_{2k}^{(3)}$ generated by $q_k^{(1)} + T_{2k}^{(3)}$. Indeed, in this case $m + T_{2k}^{(3)}$ belongs to the $T$-subspace generated by $q_k^{(0)} + T_{2k}^{(3)}$ and the latter $T$-subspace is contained in the $T$-subspace generated by $q_k^{(1)} + T_{2k}^{(3)}$ because $q_k^{(0)}$ is equal to the multilinear component of $q_k^{(1)}(1+x_1, \dots , 1+x_{2k})$. It follows that, again, $m + T_{2k}^{(3)}\in U/T_{2k}^{(3)}$. This proves our claim.

It follows that $h_2 + T^{(3)}_{n} \in U/T^{(3)}_{n}$ and,
therefore, $h + T^{(3)}_{n} \in U/T^{(3)}_{n}$, as required.

Thus, $C_n \subseteq U$ for each $n$. This completes the proof of
Proposition \ref{generators_of_C_n/T3n}.

\subsection*{Proof of Proposition \ref{generators_of_C_n}}
It is clear that the polynomial $x_1 [x_2,x_3,x_4] x_5$ generates
$T^{(3)}$ as a $T$-subspace in $F \langle X \rangle$. Since
$g_1 [g_2,g_3,g_4] g_5=g_1 [g_2,g_3,g_4, g_5] + g_1 g_5
[g_2,g_3,g_4]$
for all
$g_i \in F \langle X \rangle,$
the polynomial $x_1 [x_2,x_3,x_4]$ generates $T^{(3)}$ as a
$T$-subspace in $F \langle X \rangle$ as well. It follows that
$x_1 [x_2,x_3,x_4]$ generates $T_n^{(3)}$ as a $T$-subspace in $F
\langle X_n \rangle$ for each $n \ge 4$. Proposition
\ref{generators_of_C_n} follows immediately from Proposition
\ref{generators_of_C_n/T3n} and the observation above.

\section{Proof of Theorem \ref{theorem_main2} }

If $n = 2k +1$, $k > 1$, then Theorem \ref{theorem_main2} follows
immediately from Proposition \ref{generators_of_C_n}.

Suppose that $n = 2k$, $k \ge 1$. Then Theorem \ref{theorem_main2}
is an immediate consequence of the following two propositions.

\begin{proposition} \label{C2k_0}
For all $k \ge 1$, $C_{2k}$ is not finitely generated as a
$T$-subspace in $F \langle X_{2k} \rangle$.
 \end{proposition}

\begin{proposition}\label{C_2k_limit}
Let $k \ge 1$ and let $W$ be a $T$-subspace of $F \langle X_{2k}
\rangle$ such that $C_{2k} \subsetneqq W$. Then $W$ is a finitely
generated $T$-subspace in $F \langle X_{2k} \rangle$.
\end{proposition}

\subsection*{Proof of Proposition \ref{C2k_0}}

The proof is based on a result of Grishin and Tsybulya
\cite[Theorem 3.1]{GrishCyb}.

By Proposition \ref{generators_of_C_n/T3n}, $C_{2k}$ is generated
as a $T$-subspace in $F \langle X_{2k} \rangle$ by $T_{2k}^{(3)}$
together with the polynomials
\begin{equation}\label{gen_C2k}
x_1^p,\  x_1^p q_1 (x_2,x_3), \ \ldots , \ x_1^p q_{k-1}(x_2,
\ldots , x_{2k-1})
\end{equation}
and
\[
\{ q_k^{(l)} (x_1, \ldots , x_{2k}) \mid l = 1, 2, \ldots \} .
\]
Let $V_l$ be the $T$-subspace of $F \langle X_{2k} \rangle $
generated by $T_{2k}^{(3)}$ together with the polynomials
(\ref{gen_C2k}) and the polynomials $ \{ q_k^{(i)}(x_1, \ldots ,
x_{2k}) \mid i \le l \} $. Then we have
\begin{equation}\label{union_wl}
C_{2k} = \bigcup_{l \ge 1} V_l.
\end{equation}
Also, it is clear that $V_1 \subseteq V_2 \subseteq \ldots .$

Let $U^{(k-1)}$ be the $T$-subspace in $F \langle X \rangle$
generated by the polynomials (\ref{gen_C2k}). The following
proposition is a particular case of \cite[Theorem 3.1]{GrishCyb}.
\begin{proposition}[\cite{GrishCyb}]\label{G-Ts}
For each $l \ge 1$,
\[
(Q^{(k,l+1)}+T^{(3)}) /T^{(3)} \not \subseteq (U^{(k-1)}+
Q^{(k,l)} + T^{(3,k+1)})/T^{(3)}.
\]
\end{proposition}

\noindent \textbf{Remark. } The $T$-subspaces $(U^{(k-1)} +
T^{(3)}) / T^{(3)}$, $(Q^{(k,l)}+T^{(3)})/T^{(3)}$ and
$T^{(3,k+1)}/T^{(3)}$ are denoted in \cite{GrishCyb} by
$\sum_{i<k} CD_p^{(i)}$, $C_{p^l}^{(k)}$ and $C^{(k+1)}$,
respectively.

\medskip

Since the $T$-subspace $Q^{(k,l+1)}$ is generated by the
polynomial $q_k^{(l+1)} $ and $T^{(3)} \subset T^{(3,k+1)}$,
Proposition \ref{G-Ts} immediately implies that
\[
q_k^{(l+1)} \notin U^{(k-1)} + Q^{(k,l)} + T^{(3,k+1)}.
\]
Further, since $T_{2k}^{(3)} \subset T^{(3)} \subset  T^{(3,k+1)}$, we have
\[
V_l \subset U^{(k-1)} + \sum_{i \le l} Q^{(k,i)} + T^{(3,k+1)} =
U^{(k-1)} + Q^{(k,l)} + T^{(3,k+1)}
\]
(recall that, by (\ref{sum_q}), $\sum_{i \le l} Q^{(k,i)} =
Q^{(k,l)}$). It follows that  $q_k^{(l+1)} \notin V_l$ for all $l
\ge 1$; on the other hand, $q_k^{(l+1)} \in V_{l+1}$ by the
definition of $V_{l+1}$. Hence,
\begin{equation}\label{strictly_ascending_wl}
V_1 \subsetneqq V_2  \subsetneqq \ldots .
\end{equation}
It follows immediately  from (\ref{union_wl}) and
(\ref{strictly_ascending_wl}) that $C_{2k}$ is not finitely
generated as a $T$-subspace in $F \langle X_{2k} \rangle $. The
proof of Proposition \ref{C2k_0} is completed.

\subsection*{Proof of the Proposition \ref{C_2k_limit}}

For all integers $i_1, \ldots , i_t$ such that $1 \leq i_1<\ldots <i_t \leq n$ and all integers $a_1, \ldots , a_n \ge 0$ such that $a_{i_1}, \ldots , a_{i_t} \ge 1$, define  $\frac{x_1^{a_1} x_2^{a_2} \ldots x_n^{a_n}}{x_{i_1} x_{i_2} \ldots x_{i_t}}$ to be the monomial
 \[
 \frac{x_1^{a_1} x_2^{a_2} \ldots x_n^{a_n}}{x_{i_1}x_{i_2}\ldots x_{i_t}}=  x_1^{b_1}x_2^{b_2}\ldots x_n^{b_n} \in F \langle X \rangle,
 \]
where $b_j=a_j-1$ if $j\in \{i_1,i_2,\ldots,i_t\}$ and $b_j=a_j$ otherwise.

\begin{lemma} \label{WC(G)}
Let $f(x_1,\dots,x_n) \in F \langle X \rangle$ be a multihomogeneous polynomial of the form
\begin{equation}\label{f_modulo_T3}
f=\alpha \, x_1^{a_1} \ldots x_n^{a_n}+\sum \limits_{1 \leq i_1<\ldots <i_{2t} \leq n }
\alpha_{(i_1,\ldots,i_{2t})}\frac{x_1^{a_1}\ldots x_n^{a_n}}{x_{i_1}\ldots x_{i_{2t}}}[x_{i_1},x_{i_2}]\ldots
[x_{i_{2t-1}},x_{i_{2t}}]
\end{equation}
where $\alpha, \alpha_{(i_1,\ldots,i_{2t})}\in F$. Let $L = \langle f \rangle^{TS} + \langle [x_1,x_2] \rangle ^{TS}  +T^{(3)}$.

Suppose that $a_i = 1$ for some $i$, $1 \le i \le n.$  Then either $L = F \langle X \rangle$ or $L = \langle [x_{1},x_2] \rangle^{TS} + T^{(3)}$ or $L = \langle x_1 [x_2,x_3] \ldots [x_{2 \theta},x_{2 \theta + 1}] \rangle^{TS} +  \langle [x_{1},x_2] \rangle^{TS} + T^{(3)}$ for some $\theta \leq \frac{n-1}{2}$.
\end{lemma}

\begin{proof}
Note that each multihomogeneous polynomial $f(x_1,\dots,x_n) \in F \langle X \rangle$ can be written, modulo $T^{(3)}$, in the form (\ref{f_modulo_T3}). Hence, we can assume without loss of generality (permuting the free generators $x_1, \ldots , x_n$ if necessary) that $a_1=1$.

Note that if $\alpha \neq 0$, then $f(x_1,1,\ldots,1)=\alpha x_1 \in L$ so $ L = \langle x_1 \rangle^{TS} = F \langle X \rangle$. Suppose that $\alpha =0$.

We claim that we may assume without loss of generality that $f$ is of the form $f(x_1, \ldots, x_n) = x_1 g(x_2, \ldots, x_n),$ where
\begin{equation}\label{form_of_g}
g=\sum \limits_{\mathop{2 \le i_1<\ldots <i_{2t}\le n}\limits_{t \ge 1}}
\alpha_{(i_1,\ldots,i_{2t})}\frac{x_2^{a_2}\ldots x_n^{a_n}}{x_{i_1}\ldots x_{i_{2t}}}[x_{i_1},x_{i_2}]\ldots [x_{i_{2t-1}},x_{i_{2t}}].
\end{equation}
Indeed, consider a term $m =  \frac{x_1^{a_1}\ldots
x_n^{a_n}}{x_{i_1}\ldots x_{i_{2t}}}[x_{i_1},x_{i_2}]\ldots
[x_{i_{2t-1}},x_{i_{2t}}]$ in (\ref{f_modulo_T3}). If $i_1 >1$
then
\begin{equation}\label{m'}
m = x_1 \, \frac{x_2^{a_2}\ldots x_n^{a_n}}{x_{i_1}\ldots x_{i_{2t}}}[x_{i_1},x_{i_2}]\ldots [x_{i_{2t-1}},x_{i_{2t}}].
\end{equation}
Suppose that $i_1 =1$; then $m = m'[x_1,x_{i_2}]\ldots [x_{i_{2t-1}},x_{i_{2t}}]$, where $m' = \frac{x_2^{a_2}\ldots x_n^{a_n}}{x_{i_2}\ldots x_{i_{2t}}}$. We have
\begin{eqnarray*}
&&m + T^{(3)} = m'[x_1,x_{i_2}]\ldots [x_{i_{2t-1}},x_{i_{2t}}]+T^{(3)} \\
&& = [m' x_1,x_{i_2}]\ldots [x_{i_{2t-1}},x_{i_{2t}}]-x_1
[m',x_{i_2}]\ldots [x_{i_{2t-1}},x_{i_{2t}}]+T^{(3)} \\
&& = [m' x_1 [x_{i_3},x_{i_4}]\ldots [x_{i_{2t-1}},x_{i_{2t}}],x_{i_2}] -x_1
[m',x_{i_2}]\ldots [x_{i_{2t-1}},x_{i_{2t}}]+T^{(3)}.
\end{eqnarray*}
Hence,
\begin{equation}\label{m''}
m = - x_1 [m',x_{i_2}]\ldots [x_{i_{2t-1}},x_{i_{2t}}] + h,
\end{equation}
where $h \in \langle [x_1,x_2] \rangle^{TS} + T^{(3)}$.

It follows easily from  (\ref{m'}) and (\ref{m''}) that there
exists a multihomogeneous polynomial  $g_1 = g_1(x_2, \ldots, x_n)
\in F \langle X \rangle$ such that $f = x_1 g_1 + h_1$, where $h_1
\in \langle [x_1,x_2] \rangle^{TS} + T^{(3)}$. Further, there is a
multihomogeneous polynomial $g$ of the form (\ref{form_of_g}) such
that $g \equiv g_1 \pmod{T^{(3)}}$; then $f = x_1 g + h_2$, where
$h_2 \in \langle [x_1,x_2] \rangle^{TS} + T^{(3)}$. It follows
that $L = \langle x_1g(x_2, \ldots , x_n) \rangle^{TS} + \langle
[x_1,x_2] \rangle^{TS} + T^{(3)}$. Thus, we can assume without
loss of generality that  $f = x_1 g(x_2, \ldots , x_n)$, where $g$
is of the form (\ref{form_of_g}), as claimed.

If $f = 0$ then $L = \langle [x_{1},x_2] \rangle^{TS} + T^{(3)}$.
Suppose that $f \ne 0$. Let $\theta = \mbox{ min }\{ t \mid
\alpha_{(i_1,\ldots,i_{2t})} \neq 0 \}$. It is clear that $2
\theta + 1 \le n$ so $\theta \le \frac{n-1}{2}$. We can assume
that $\alpha_{(2,\ldots,2\theta+1)}\neq 0$; then
\begin{eqnarray}\label{f_final_form}
f &=& x_1 \Big( \alpha_{(2, \ldots , 2 \theta +1)}
\frac{x_2^{a_2} \ldots x_n^{a_n}}{ x_2 \ldots x_{2 \theta
+1}}[x_{2},x_{3}]\ldots
[x_{2 \theta},x_{2 \theta +1}] \nonumber\\
&+& \sum \limits_{\mathop{2 \leq i_1<\ldots <i_{2t} \leq n }
\limits_{t \ge \theta, \ i_{2t} > 2 \theta +1}}
\alpha_{(i_1,\ldots,i_{2t})}  \frac{x_2^{a_2} \ldots
x_n^{a_n}}{x_{i_1}\ldots x_{i_{2t}}}[x_{i_1},x_{i_2}]\ldots
[x_{i_{2t-1}},x_{i_{2t}}] \Big).
\end{eqnarray}
Let $f_1(x_1, \ldots , x_{2 \theta + 1}) = f(x_1,x_2, \ldots,
x_{2\theta+1},1,\ldots,1) \in L$; then
\[
f_1=  \alpha_{(2, \ldots , 2 \theta +1)} \ x_1 \frac{x_2^{a_2}
\ldots x_n^{a_n}}{ x_2 \ldots x_{2 \theta +1}}[x_{2},x_{3}]\ldots
[x_{2 \theta},x_{2 \theta +1}].
\]
It can be easily seen that the multihomogeneous component of
degree 1 in the variables $x_1,x_2,\ldots, x_{2\theta+1}$ of the
polynomial $f_1(x_1,x_2+1, \ldots, x_{2\theta+1}+1)$ is equal to
\[
\alpha_{(2, \ldots , 2 \theta +1)} \, x_1[x_2,x_3]\ldots
[x_{2\theta}, x_{2\theta+1}].
\]
It follows that $x_1[x_2,x_3]\ldots [x_{2\theta}, x_{2\theta+1}]
\in L$; hence,
\[
\langle x_1 [x_2,x_3] \ldots [x_{2 \theta},x_{2 \theta + 1}]
\rangle^{TS} +  \langle [x_{1},x_2] \rangle^{TS} + T^{(3)}
\subseteq L.
\]

On the other hand, it is clear that the polynomial $f$ of the form
(\ref{f_final_form}) belongs to the $T$-subspace of $F \langle X
\rangle$ generated by $x_1[x_2,x_3]\ldots [x_{2\theta},
x_{2\theta+1}]$; it follows that $\langle f \rangle^{TS} \subseteq
\langle x_1 [x_2,x_3] \ldots [x_{2 \theta},x_{2 \theta + 1}]
\rangle^{TS}$ and, therefore,
\[
L \subseteq \langle x_1 [x_2,x_3] \ldots [x_{2 \theta},x_{2 \theta
+ 1}] \rangle^{TS} +  \langle [x_{1},x_2] \rangle^{TS} + T^{(3)}.
\]

Thus, $L = \langle x_1 [x_2,x_3] \ldots [x_{2 \theta},x_{2 \theta
+ 1}] \rangle^{TS} +  \langle [x_{1},x_2] \rangle^{TS} + T^{(3)}$.
The proof of Lemma \ref{WC(G)} is completed.
\end{proof}

\begin{proposition}
\label{C_2k/T32k_limit} Let $W$ be a $T$-subspace of $F \langle X_{2k} \rangle$ such that $C_{2k} \subsetneqq W$. Then $W=F \langle X_{2k} \rangle$ or $W$ is generated as a $T$-subspace by the polynomials
\[
x_1^p, \ x_1^pq_1(x_2,x_3), \ldots, \ x_1^pq_{\lambda -1}(x_2,\ldots,x_{2\lambda-1}),
\]
\[
x_1[x_2,x_3,x_4], \  x_1[x_2,x_3]\ldots [x_{2\lambda},x_{2\lambda+1}],
\]
for some positive integer $\lambda \leq k-1$.
\end{proposition}

\begin{proof}
It is well-known that over a field $F$ of characteristic $0$ each $T$-ideal in $F \langle X \rangle$ can be generated by its multilinear polynomials. It is easy to check that the same is true for each $T$-subspace in $F \langle X \rangle$. Over an infinite field $F$ of characteristic $p>0$ each $T$-ideal in $F \langle X \rangle$ can be generated  by all its multihomogeneous polynomials $f(x_1, \ldots , x_n)$ such that, for each $i$, $ 1 \le i \le n$, $\mbox{deg}_{x_i} f = p^{s_i}$ for some integer $s_i$ (see, for instance, \cite{Baht}). Again, the same is true for each $T$-subspace in $F \langle X \rangle$.

Let $f(x_1,\ldots,x_{2k})\in W \setminus C_{2k}$ be an arbitrary multihomogeneous polynomial such that, for each $i$ ($1\leq i \leq 2k$), we have either $\mbox{deg}_{x_i} f=p^{s_i}$ or $\mbox{deg}_{x_i} f=0$. We may assume that $\mbox{deg}_{x_i} f=p^{s_i}$ for $i = 1, \ldots , l$ and $\mbox{deg}_{x_i} f=0$ for $i = l+1, \ldots , 2k$ (that is, $f = f(x_1, \ldots , x_l$)).  Then we have
 \[
 f+T^{(3)}_{2k}=\alpha \, m+\sum \limits_{1 \le i_1<\ldots <i_{2t} \le l}
\alpha_{(i_1,\ldots,i_{2t})}\frac{m}{x_{i_1}\ldots x_{i_{2t}}}[x_{i_1},x_{i_2}]\ldots [x_{i_{2t-1}},x_{i_{2t}}]+T^{(3)}_{2k},
 \]
where $\alpha, \alpha_{(i_1,\ldots,i_{2t})}\in F$,  \  $m=x_1^{p^{s_1}} \ldots x_{l}^{p^{s_{l}}}.$

If $s_i > 0$ for all $i = 1, \ldots , l$ then it can be easily seen that $f \in C(G)$ so $f \in C_{2k}$, a contradiction with the choice of $f$. Thus, $s_i =0$ for some $i$, $1 \le i \le l$. Let $L_f$ be the $T$-subspace of $F \langle X \rangle$ generated by $f$, $[x_1,x_2]$ and $T^{(3)}$. By Lemma \ref{WC(G)}, we have either $L_f = F\langle X \rangle$ or
\[
L_f = \langle x_1 [x_2,x_3]\ldots [x_{2\theta},x_{2\theta+1}] \rangle^{TS}  + \langle [x_1,x_2] \rangle^{TS} + T^{(3)}
\]
for some $\theta < k$ (since $f \notin C_{2k}$, we have $L_f \ne \langle [x_1,x_2] \rangle^{TS} + T^{(3)}$). Note that if $k=1$ (that is, $f=f(x_1,x_2)$) then the only possible case is $L_f = F\langle X \rangle$.

It is clear that if $L_f = F \langle X \rangle$ for some $f \in W \setminus C_{2k}$ then $x_1 \in W$ so $W = F \langle X_{2k} \rangle$. Suppose that $W \ne F \langle X_{2k} \rangle$; then $k>1$ and $L_f \ne F \langle X \rangle $ for all $f \in W \setminus C_{2k}$. For each $f \in W \setminus C_{2k}$ satisfying the conditions of Lemma \ref{WC(G)}, the $T$-subspace $L_f$ in $F \langle X \rangle$ can be generated, by Lemma \ref{WC(G)}, by the polynomials
\begin{equation}\label{polyn_generators}
[x_1,x_2], \quad  x_1[x_2,x_3x_4]\quad \mbox{and} \quad x_1 [x_2,x_3]\ldots [x_{2\theta},x_{2\theta+1}]
\end{equation}
for some $\theta= \theta_f < k$. Since the polynomials (\ref{polyn_generators}) belong to $F \langle X_{2k} \rangle$ (recall that $k>1$), the $T$-subspace in $F \langle X_{2k} \rangle$ generated by $f$, $[x_1,x_2]$ and $ T^{(3)}$ is also generated (as a $T$-subspace in $F \langle X_{2k} \rangle$) by the polynomials (\ref{polyn_generators}). Note that $[x_1,x_2]$ and $x_1[x_2, x_3, x_4]$ belong to $C_{2k}$ so the $T$-subspace $V_f$ in $F \langle X_{2k} \rangle$ generated by $f$ and $C_{2k}$ can be generated by $C_{2k}$ and $x_1 [x_2,x_3]\ldots [x_{2\theta},x_{2\theta+1}]$ for some $\theta= \theta_f < k$.

Let $\lambda = \mbox{ min } \{ \theta \mid x_1 [x_2,x_3]\ldots [x_{2\theta},x_{2\theta+1}] \in W \}$. Since $W$ is the sum of the $T$-subspaces $V_f$ for all suitable multihomogeneous polynomials $f \in W \setminus C_{2k}$ and each $V_f$ is generated by $C_{2k}$ and $x_1 [x_2,x_3]\ldots [x_{2\theta},x_{2\theta+1}]$ for some $\theta = \theta_f < k$, $W$ can be generated as a $T$-subspace in $F \langle X_{2k} \rangle$ by $C_{2k}$ and $x_1 [x_2,x_3]\ldots [x_{2\lambda},x_{2\lambda+1}]$. Now it follows easily from Proposition \ref{generators_of_C_n} that $W$ can be generated by the polynomials
\[
x_1^p, \ x_1^p q_1(x_2,x_3), \ldots, \ x_1^p q_{\lambda -1}(x_2,\ldots,x_{2\lambda-1})
\]
together with the polynomials
\[
x_1[x_2,x_3,x_4] \ \mbox{ and } \ x_1[x_2,x_3] \ldots [x_{2\lambda},x_{2\lambda+1}],
\]
where we note that $\lambda <k$.

This completes the proof of Proposition \ref{C_2k/T32k_limit}.
\end{proof}

Proposition \ref{C_2k_limit} follows immediately from Proposition
\ref{C_2k/T32k_limit}. The proof of Theorem \ref{theorem_main2} is
completed.

\section{Proof of Theorem \ref{theorem_main3} }

\begin{proposition} \label{rk}
For each $k \ge 1$, $R_k$  is not finitely generated as a
$T$-subspace in $F \langle X \rangle$.
\end{proposition}
\begin{proof}
Recall that $R_k$ is the  $T$-subspace in $F \langle X \rangle$
generated by $C_{2k}$ and $T^{(3,k+1)}$. By Proposition
\ref{generators_of_C_n/T3n}, $C_{2k}$ is generated as a
$T$-subspace in $F \langle X_{2k} \rangle$ by $T_{2k}^{(3)}$
together with the polynomials (\ref{gen_C2k}) and the polynomials
$\{ q_k^{(l)} (x_1, \ldots ,x_{2k}) \mid l = 1, 2, \ldots \} .$
Since $T_{2k}^{(3)} \subset T^{(3)} \subset T^{(3,k+1)}$, we have
\[
R_k=U^{(k-1)}+ \sum_{l \ge 1} Q^{(k,l)}+T^{(3, k+1)},
\]
where $U^{(k-1)}$ and $Q^{(k,l)}$  are the $T$-subspaces in $F
\langle X \rangle$ generated by the polynomials  (\ref{gen_C2k})
and by the polynomial $q_k^{(l)}(x_1, \ldots , x_{2k})$,
respectively.

Let $V_l = U^{(k-1)}+ \sum_{i \le l} Q^{(k,i)}+T^{(3, k+1)}$. Then
\begin{equation}\label{union_rk}
R_k = \bigcup_{l \ge 1} V_l
\end{equation}
and $V_1 \subseteq V_2  \subseteq \ldots .$ Recall that, by
(\ref{sum_q}), $\sum_{i \le l} Q^{(k,i)} = Q^{(k,l)}$ so $V_l =
U^{(k-1)}+ Q^{(k,l)}+T^{(3, k+1)}$. By Proposition \ref{G-Ts},
$Q^{(k,l+1)} \not \subseteq V_l$ for all $l \ge 1$ so
\begin{equation}\label{strictly-ascending_rk}
V_1 \subsetneqq V_2 \subsetneqq \ldots .
\end{equation}
The result follows immediately from (\ref{union_rk}) and
(\ref{strictly-ascending_rk}).
\end{proof}

\begin{lemma}\label{fTSRk}
Let $f = f(x_1, \ldots,x_{n}) \in F \langle X \rangle$ be a
multihomogeneous polynomial of the form
\begin{equation}\label{f_modulo_T3_2}
f=\alpha \, x_1^{p^{s_1}}  \ldots x_n^{p^{s_n}} + \sum_{i_1<
\ldots <i_{2t}} \alpha_{(i_1,\ldots,i_{2t})} \,
\frac{x_1^{p^{s_1}} \ldots x_n^{p^{s_n}}}{x_{i_1}\ldots
x_{i_{2t}}}[x_{i_1},x_{i_2}]\ldots [x_{i_{2t-1}},x_{i_{2t}}],
\end{equation}
where $\alpha,  \alpha_{(i_1,\ldots,i_{2t})}\in F$, $s_i \geq 0$
for all $i$. Let $L = \langle f \rangle^{TS} + R_k$, $k \ge 1$.
Then one of the following holds:
\begin{enumerate}
\item $L = F \langle X \rangle$;
\item $L = R_k$;
\item $L = \langle x_1[x_2,x_3]  \ldots [x_{2 \theta},x_{2 \theta
+1}] \rangle^{TS} + R_k$ for some $\theta$, $1 \le \theta \le k$;
\item $L=\langle  x_1^{p^s}q_k^{(s)}(x_2,\ldots,x_{2k+1}) \rangle
^{TS}+R_k$ for some $s \ge 1$.
\end{enumerate}
\end{lemma}

\begin{proof}
Note that each multihomogeneous  polynomial $f(x_1,\dots,x_n) \in
F \langle X \rangle$ of degree $p^{s_i}$ in $x_i$ ($1 \le i \le
n$) can be written, modulo $T^{(3)}$, in the form
(\ref{f_modulo_T3_2}). Hence, we can assume without loss of
generality (permuting the free generators $x_1, \ldots , x_n$ if
necessary) that $s_1 \le s_i$ for all $i$. Write $s = s_1$.

Suppose that $s=0$. Then, by Lemma \ref{WC(G)}, we have either
\[
\langle f \rangle^{TS} +  \langle [x_1,x_2] \rangle^{TS} + T^{(3)}
= F \langle X \rangle
\]
or
\[
\langle f \rangle^{TS} +  \langle [x_1,x_2] \rangle^{TS} + T^{(3)}
= \langle [x_1,x_2] \rangle^{TS} + T^{(3)}
\]
or
\[
\langle f \rangle^{TS} +  \langle [x_1,x_2] \rangle^{TS} + T^{(3)}
= \langle x_1[x_2,x_3] \ldots [x_{2 \theta},x_{2 \theta +1}]
\rangle^{TS} + \langle [x_1,x_2] \rangle^{TS} + T^{(3)}
\]
for some $\theta$. Since  $\langle [x_1,x_2] \rangle^{TS} +
T^{(3)} \subset R_k$ and $x_1[x_2,x_3] \ldots [x_{2 \theta},x_{2
\theta +1}] \in R_k$ if $\theta > k$, we have either $L = F
\langle X \rangle$ or $L = R_k$ or
\[
L = \langle x_1[x_2,x_3]  \ldots [x_{2 \theta},x_{2 \theta +1}]
\rangle^{TS} + R_k
\]
for some $\theta \le k$.

Now suppose that $s>0$;  then $s_i >0$ for all $i$, $1 \leq i \leq
n$. It can be easily seen that, by (\ref{prop2}), $  x_1^{p^{s_1}} \ldots
x_n^{p^{s_n}} \in \left(\langle x_1^p \rangle^{TS} + T^{(3)}
\right) \subset R_k$ and, for all $t <k$,
\[
\frac{x_1^{p^{s_1}}  \ldots x_n^{p^{s_n}}}{x_{i_1}\ldots
x_{i_{2t}}}[x_{i_1},x_{i_2}]\ldots [x_{i_{2t-1}},x_{i_{2t}}] \in
\left(\langle x_1^pq_t(x_2, \ldots , x_{2t+1})\rangle^{TS} +
T^{(3)}\right) \subset R_k.
\]
Also we have  $\frac{x_1^{p^{s_1}} \ldots
x_n^{p^{s_n}}}{x_{i_1}\ldots x_{i_{2t}}}[x_{i_1},x_{i_2}]\ldots
[x_{i_{2t-1}},x_{i_{2t}}] \in T^{(3,k+1)} \subset R_k$ for each $t
> k$.  It follows that we can assume without loss of generality
that the polynomial $f$ is of the form
\begin{equation}\label{form_f}
f = \sum \limits_{1 \le i_1<\ldots <i_{2k} \le n}
\alpha_{(i_1,\ldots,i_{2k})} \,\frac{x_1^{p^{s_1}} \ldots
x_n^{p^{s_n}}}{x_{i_1}\ldots x_{i_{2k}}}[x_{i_1},x_{i_2}]\ldots
[x_{i_{2k-1}},x_{i_{2k}}].
\end{equation}

Note that if $n < 2k$ then $f=0$ and if $n=2k$ then
\[
f = \alpha_{(1,2, \ldots , 2k)}  \frac{x_1^{p^{s_1}} \ldots
x_{2k}^{p^{s_{2k}}}}{x_{1}x_2\ldots x_{2k}}[x_{1},x_{2}]\ldots
[x_{2k-1},x_{2k}]
\]
so, by Lemma \ref{lemma_GTs},  we have $f \in Q^{(k,s)} +
T^{(3)}$, where $s=s_1>0$. In both cases we have $f \in R_k$ and
$L=R_k$.

Suppose that $n > 2k$. We claim  that we may assume that $f$ is of
the form
\begin{equation}\label{form_f2}
f(x_1, \ldots , x_n) = x_1^{p^s} g(x_2, \ldots , x_n),
\end{equation}
where
\[
g= \sum \limits_{2 \le i_1<\ldots <i_{2k} \le n}
\alpha_{(i_1,\ldots,i_{2k})} \, \frac{x_2^{p^{s_2}} \ldots
x_n^{p^{s_n}}}{x_{i_1}\ldots x_{i_{2k}}}[x_{i_1},x_{i_2}]\ldots
[x_{i_{2k-1}},x_{i_{2k}}].
\]

Indeed, consider  a term $m = \frac{x_1^{p^{s_1}} \ldots
x_n^{p^{s_n}}}{x_{i_1}\ldots x_{i_{2k}}}[x_{i_1},x_{i_2}]\ldots
[x_{i_{2k-1}},x_{i_{2k}}]$ in (\ref{form_f}). If $i_1 >1$ then
\begin{equation}\label{m_x1_1}
m = x_1^{p^{s}} \frac{x_2^{p^{s_2}}  \ldots
x_n^{p^{s_n}}}{x_{i_1}\ldots x_{i_{2k}}}[x_{i_1},x_{i_2}]\ldots
[x_{i_{2k-1}},x_{i_{2k}}].
\end{equation}

Suppose that $i_1 = 1$. Let $a_i = p^{s_i}$ for all $i$. Then
\[
m +T^{(3,k+1)} = x_1^{p^{s}-1}  \frac{x_2^{p^{s_2}} \ldots
x_n^{p^{s_n}}}{x_{i_2}\ldots x_{i_{2k}}}[x_{1},x_{i_2}]\ldots
[x_{i_{2k-1}},x_{i_{2k}}] +T^{(3,k+1)}
\]

\[
= x_{j_1}^{a_{j_1}}  \cdots x_{j_{l}}^{a_{j_{l}}} x_{1}^{a_{1}-1}
\cdots x_{i_{2k}}^{a_{i_{2k}}-1} [x_1,x_{i_{2}}] \cdots
[x_{i_{2k-1}},x_{i_{2k}}] + T^{(3,k+1)}
\]

\[
= x_1^{a_1-1}x_{j_1}^{a_{j_1}} \ldots
x_{j_l}^{a_{j_l}}[x_1,x_{i_2}]x_{i_2}^{a_{i_2}-1} m'+T^{(3,k+1)},
\]
where
\[
m'=x_{i_{3}}^{a_{i_{3}}-1} [x_{i_{3}},x_{i_{4}}]
x_{i_{4}}^{a_{i_{4}}-1} \ldots x_{i_{2k-1}}^{a_{i_{2k-1}}-1}
[x_{i_{2k-1}},x_{i_{2k}}] x_{i_{2k}}^{a_{i_{2k}}-1},
\]
$\{j_1,\ldots,j_l\}=\{1,\ldots,n\}  \setminus
\{1,i_2,\ldots,i_{2k}\}$, $l=n-2k>0$. Suppose that
\[
a_1=a_{j_1}=a_{j_2}=\ldots=a_{j_z} \  \ \mbox{and} \ \
a_{j_{z+1}},a_{j_{z+2}},\ldots,a_{j_l}>a_1.
\]
Let
\[
u=x_1x_{j_1} \cdots x_{j_z}x_{j_{z+1}}^{a'_{j_{z+1}}}\cdots
x_{j_{l}}^{a'_{j_{l}}},
\]
where $a'_i = a_i/p^s$ for all $i$. Let
\[
h= h(x_1,  \ldots ,x_{2k}) = x_1^{a_{1}-1}[x_1,x_2]x_2^{a_{i_2}-1} \ldots
x_{2k-1}^{a_{i_{2k-1}}-1}[x_{2k-1},x_{2k}] x_{2k}^{a_{i_{2k}-1}}.
\]
By (\ref{prop}), $h \in C(G)$; hence,  $h \in C_{2k} \subset R_k$.
It follows that $h(u,x_{i_2}, \ldots , x_{i_{2k}}) \in R_k$, that
is,
\begin{equation}\label{uxi2}
u^{p^s-1}[u,x_{i_2}]x_{i_2}^{a_{i_2}-1} m'\in R_k.
\end{equation}
Since, by (\ref{prop2}), $[v_1^p,v_2]  \in T^{(3)} \subset
T^{(3,k+1)}$ for all $v_1,v_2 \in F \langle X \rangle$, we have
\begin{eqnarray*}
&&u^{p^s-1}[u,x_{i_2}]x_{i_2}^{a_{i_2}-1} m' +T^{(3,k+1)}\\
&&= \left(x_1x_{j_1} \cdots
x_{j_z}\right)^{p^s-1}x_{j_{z+1}}^{a_{j_{z+1}}}\cdots
x_{j_l}^{a_{j_l}}[x_1x_{j_1} \ldots
x_{j_z},x_{i_2}]x_{i_2}^{a_{i_2}-1} m' +T^{(3,k+1)} \\
&& = \left(x_1 x_{j_1} \cdots
x_{j_z}\right)^{p^s-1} x_{j_{z+1}}^{a_{j_{z+1}}}\cdots
x_{j_l}^{a_{j_l}}[x_1,x_{i_2}]x_{j_1} \ldots x_{j_z}
x_{i_2}^{a_{i_2}-1} m' \\
&& + \left(x_1x_{j_1} \cdots
x_{j_z}\right)^{p^s-1}x_{j_{z+1}}^{a_{j_{z+1}}}\cdots
x_{j_l}^{a_{j_l}}x_1[x_{j_1} \ldots
x_{j_z},x_{i_2}]x_{i_2}^{a_{i_2}-1} m' + T^{(3,k+1)} \\
&&= m + x_1^{p^s} x_{j_1}^{p^s -1} \cdots x_{j_z}^{p^s-1} x_{j_{z+1}}^{a_{j_{z+1}}}\cdots x_{j_l}^{a_{j_l}} [x_{j_1} \ldots
x_{j_z},x_{i_2}]x_{i_2}^{a_{i_2}-1} m' +T^{(3,k+1)}
\end{eqnarray*}
where the second summand is not present if $z = 0$ (that is, if $a_{j_i} > a_1$ for all $i$), in which case $m \in R_k$. Since
\begin{eqnarray*}
&&x_1^{p^s} x_{j_1}^{p^s -1} \cdots x_{j_z}^{p^s-1} x_{j_{z+1}}^{a_{j_{z+1}}}\cdots x_{j_l}^{a_{j_l}} [x_{j_1} \ldots
x_{j_z},x_{i_2}]x_{i_2}^{a_{i_2}-1} m' +T^{(3,k+1)} \\
&& = x_1^{p^s} \sum \limits_{2 \le i_1<\ldots <i_{2k}}
\beta_{(i_1,\ldots,i_{2k})} \,   \frac{x_2^{p^{s_2}}
\ldots x_n^{p^{s_n}}}{x_{i_1}\ldots
x_{i_{2k}}}[x_{i_1},x_{i_2}]\ldots [x_{i_{2k-1}},x_{i_{2k}}]
+T^{(3,k+1)}
\end{eqnarray*}
for some $\beta_{(i_1,\ldots,i_{2k})} \in F$, we have
\begin{equation}\label{mi1i2k_2}
m + x_1^{p^s} \sum \limits_{2 \le i_1<\ldots <i_{2k}}
\beta_{(i_1,\ldots,i_{2k})}  \,  \frac{x_2^{p^{s_2}}
\ldots x_n^{p^{s_n}}}{x_{i_1}\ldots
x_{i_{2k}}}[x_{i_1},x_{i_2}]\ldots [x_{i_{2k-1}},x_{i_{2k}}] \in
R_k.
\end{equation}

It is clear that, using (\ref{m_x1_1}) and (\ref{mi1i2k_2}), we
can write $f = f_1 + f_2$, where
\[
f_1 = x_1^{p^s} \left( \sum \limits_{2 \le i_1<\ldots <i_{2k}}
\gamma_{(i_1,\ldots,i_{2k})}  \, \frac{x_2^{p^{s_2}} \ldots
x_n^{p^{s_n}}}{x_{i_1}\ldots x_{i_{2k}}}[x_{i_1},x_{i_2}]\ldots
[x_{i_{2k-1}},x_{i_{2k}}] \right)
\]
is of the form (\ref{form_f2}) and $f_2 \in R_k$.  Then we have
$\langle f \rangle^{TS} + R_k = \langle f_1 \rangle^{TS} + R_k.$
Thus, we can assume (replacing $f$ with $f_1$) that the polynomial
$f$ is of the form (\ref{form_f2}), as claimed.

If $f = 0$ then $L = R_k$. Suppose that $f \ne 0$.  Then we can
assume without loss of generality that
$\alpha_{(2,3,\ldots,2k+1)}\neq 0$. It follows that the
$T$-subspace $\langle f \rangle^{TS}$ contains the polynomial
\begin{eqnarray*}
&&h(x_1,\ldots ,x_{2k+1})=\alpha_{(2,3,\ldots,2k+1)}^{-1}f(x_1,\ldots ,x_{2k+1},1,1,\ldots,1)\\
&&= x_1^{p^s}x_2^{p^{s_2}-1}\ldots x_{2k+1}^{p^{s_{2k+1}}-1}[x_2,x_3]\ldots [x_{2k},x_{2k+1}].
\end{eqnarray*}
Then $\langle f \rangle^{TS} + R_k$  also contains the homogeneous
component of the polynomial $h(x_1+1,\ldots,x_{2k+1}+1)$ of degree
$p^s$ in each variable $x_i$ $(i = 1,2, \ldots , 2k+1)$, that is
equal, modulo $T^{(3)}$,  to
\[
\gamma \ x_1^{p^s}x_2^{p^s-1}\ldots x_{2k+1}^{p^s-1}[x_2,x_3]
\ldots [x_{2k}, x_{2k+1}],
\]
where $\gamma = \prod_{i=2}^{2k+1} {p^{s_i}-1 \choose p^s-1}
\equiv 1 \pmod{p}$. It follows that
\[
x_1^{p^s}q_k^{(s)}(x_2,\ldots,x_{2k+1}) \in \langle f \rangle^{TS} + R_k.
\]
On the other hand,  for all $i_1, \ldots ,i_{2k}$ such that $2 \le
i_1 < \ldots < i_{2k} \le n$, we have
\[
x_1^{p^s}  \frac{x_2^{p^{s_2}} \ldots x_n^{p^{s_n}}}{x_{i_1}\ldots
x_{i_{2k}}}[x_{i_1},x_{i_2}]\ldots [x_{i_{2k-1}},x_{i_{2k}}] \in
\langle x_1^{p^s}q_k^{(s)}(x_2,\ldots,x_{2k+1}) \rangle^{TS} +
T^{(3,k+1)}
\]
(recall that $s_i \ge s$ for all $i$) so
\[
f \in \langle  x_1^{p^s}q_k^{(s)}(x_2,\ldots,x_{2k+1})
\rangle^{TS} + R_k.
\]
Thus,
\[
\langle f \rangle^{TS}+R_k=\langle
x_1^{p^s}q_k^{(s)}(x_2,\ldots,x_{2k+1}) \rangle ^{TS}+R_k,
\]
where $s \ge 1$. The proof of Lemma \ref{fTSRk} is completed.
\end{proof}

\begin{proposition} \label{Rk_gen}
Let $W$ be a $T$-subspace of $F \langle X \rangle$ such that $R_k
\subsetneqq W.$ Then one of the following holds:
\begin{enumerate}
\item $W=F\langle X \rangle$.
\item $W$ is generated as a $T$-subspace by the polynomials
\[
x_1^p, \ x_1^p q_1(x_2,x_3),  \ldots, \ x_1^p q_{\lambda-1}(x_2,
\ldots ,x_{2 \lambda -1}),
 \]
 \[
 x_1[x_2,x_3,x_4],\ x_1[x_2,x_3]\ldots [x_{2\lambda},x_{2\lambda+1}]
 \]
for some $\lambda \leq k$.
\item $W$ is generated as a $T$-subspace by the polynomials
\[
x_1^p, \ x_1^pq_1(x_2,x_3),  \ldots, \ x_1^pq_{k-1}(x_2, \ldots
,x_{2 k -1}),
\]
\[
\{q_k^{(l)}(x_1, \ldots ,x_{2k})  \mid 1\leq l \leq \mu-1 \}, \
x_1^{p^{\mu}}q_k^{(\mu)}(x_2, \ldots ,x_{2k+1}),
\]
\[
x_1[x_2,x_3,x_4], \ x_1[x_2,x_3]\ldots [x_{2k+2},x_{2k+3}]
\]
for some $\mu \ge 1$.
\end{enumerate}
\end{proposition}

\begin{proof}
Let $f = f(x_1, \ldots , x_n)$ be an arbitrary polynomial in $W
\setminus R_k$ satisfying the conditions of Lemma \ref{fTSRk},
that is, an arbitrary multihomogeneous polynomial such that
$\mbox{deg}_{x_i}f = p^{s_i}$ for some $s_i \ge 0$ $(1 \le i \le
n)$. Let $L_f = \langle f \rangle^{TS} + R_k$. By Lemma
\ref{fTSRk}, we have either $L_f = F \langle X \rangle$ or
\[
L_f = \langle x_1[x_2,x_3] \ldots [x_{2 \theta},x_{2 \theta +1}]
\rangle^{TS} + R_k
\]
for some $\theta \le k$ or
\[
L_f=\langle x_1^{p^s}q_k^{(s)}(x_2,\ldots,x_{2k+1}) \rangle ^{TS}+R_k
\]
for some $s \ge 1$.

Note that $W$ is generated as a $T$-subspace in $F \langle X
\rangle$ by $R_k$ together with the polynomials $f \in W
\setminus R_k$ satisfying the conditions of Lemma
\ref{fTSRk}. It follows that $W = \sum L_f $, where the sum
is taken over all the polynomials $f \in W \setminus R_k$
satisfying these conditions.

It is  clear that if $L_f = F \langle X \rangle$ for some $f \in W
\setminus R_k$ then $W = F \langle X \rangle$. Suppose that $L_f
\ne F \langle X \rangle$ for all $f \in W \setminus R_k$. Let, for
some $f \in W \setminus R_k$, we have $L_f = \langle x_1[x_2,x_3]
\ldots [x_{2 \theta},x_{2 \theta +1}] \rangle^{TS} + R_k$, $\theta
\le k$. Define $\lambda = \mbox{ min }\{ \theta \mid x_1[x_2,x_3]
\ldots [x_{2 \theta},x_{2 \theta +1}] \in W \}$; note that
$\lambda \le k$. We have
\[
x_1[x_2,x_3]  \ldots [x_{2 \theta},x_{2 \theta +1}] \in \langle
x_1[x_2,x_3] \ldots [x_{2 \lambda},x_{2 \lambda +1}] \rangle^{TS}
\]
for all $\theta  \ge \lambda$ and
\[
x_1^{p^s}q_k^{(s)}(x_2,\ldots,x_{2k+1})  \in \langle x_1[x_2,x_3]
\ldots [x_{2 \lambda},x_{2 \lambda +1}] \rangle^{TS} + T^{(3)}
\]
for all $s$. Hence, $W =  \langle x_1[x_2,x_3] \ldots [x_{2
\lambda},x_{2 \lambda +1}] \rangle^{TS} + R_k$, where $\lambda \le
k$. It follows that $W$ is generated as a $T$-subspace by the
polynomials
\[
x_1^p, \ x_1^p q_1(x_2,x_3),  \ldots, \ x_1^p q_{\lambda-1}(x_2,
\ldots ,x_{2 \lambda -1}),
\]
\[
x_1[x_2,x_3,x_4], \ x_1[x_2,x_3] \ldots [x_{2 \lambda}, x_{2
\lambda+1}],
\]
$\lambda \leq k$.

Now suppose that,  for all $f \in W \setminus R_k$ satisfying the
conditions of Lemma \ref{fTSRk}, we have
\[
L_f=\langle x_1^{p^s}q_k^{(s)}(x_2,\ldots,x_{2k+1}) \rangle ^{TS}+R_k
\]
for some $s = s_f  \ge 1$. Note that if $s \leq r$ then
\[
x_1^{p^r}q_k^{(r)}(x_2,  \ldots, x_{2k+1}) \in \langle
x_1^{p^s}q_k^{(s)}(x_2,\ldots,x_{2k+1}) \rangle ^{TS} + T^{(3)}.
\]
Take $\mu =  \mbox{ min } \{ s \mid
x_1^{p^s}q_k^{(s)}(x_2,\ldots,x_{2k+1}) \in W \}$. Then we have
$W=R_k+\langle x_1^{p^\mu} q_k^{(\mu)} (x_2,\ldots,x_{2k+1})
\rangle^{TS}$ and it is straightforward to check that $W$ can be
generated as a $T$-subspace in $F \langle X \rangle$ by the
polynomials
\[
x_1^p, \ x_1^pq_1(x_2,x_3),  \ldots, \
x_1^pq_{k-1}(x_2,\ldots,x_{2k-1})
\]
and the polynomials $\{ q_k^{(l)}(x_1,\ldots,x_{2k})  \mid 1\leq l
\leq \mu-1 \}$, $ x_1^{p^{\mu}} q_k^{(\mu)}(x_2, \ldots,
x_{2k+1})$ together with the polynomials
\[
x_1[x_2,x_3,x_4] \ \ \mbox{and}  \ \ x_1[x_2,x_3]\ldots
[x_{2k+2},x_{2k+3}].
\]

This completes the proof of Proposition \ref{Rk_gen}.
\end{proof}

Proposition \ref{Rk_gen} immediately implies the following corollary.

\begin{corollary} \label{Rk_limit}
Let $W$ be a $T$-subspace of $F \langle  X \rangle$ such that $R_k
\subsetneqq W$  ($k \ge 1$). Then $W$ is a finitely generated
$T$-subspace in $F \langle X \rangle$.
\end{corollary}

\begin{proposition} \label{Not_equal}
If $k \ne l$ then $R_k \ne R_l.$
\end{proposition}

\begin{proof}
Suppose, in order to get a contradiction, that $R_k = R_l$ for
some $k,l$, $k < l$. Then we have $C(G) \subseteq R_l.$

Indeed, by Theorem \ref{generators_of_C}, the $T$-subspace $C(G)$
is generated by the polynomial $x_1[x_2,x_3,x_4]$ and the
polynomials $x_1^p, x_1^p q_1(x_2,x_3),  \ldots , x_1^p q_n(x_2,
\ldots , x_{2n+1}), \ldots .$ Clearly,
\[
x_1[x_2,x_3,x_4] \in T^{(3)} \subset R_l.
\]
Further,
\[
x_1^p, \  x_1^p q_1(x_2,x_3), \ \ldots , \  x_1^p q_{l-1}(x_2,
\ldots , x_{2l-1}) \in R_l
 \]
by the definition of $R_l$ and
\[
x_1^p q_{k+1}(x_2, \ldots , x_{2k+3}), \ x_1^p q_{k+2}(x_2, \ldots
, x_{2k+5}), \ldots \in T^{(3, k+1)} \subseteq R_k = R_l
 \]
by the definition of $T^{(3,k+1)}$. Since $k < l$, we have
\[
x_1^p, \  x_1^p q_1(x_2,x_3), \  \ldots , \ x_1^p q_{k}(x_2,
\ldots , x_{2k+1}), \ x_1^p q_{k+1}(x_2, \ldots , x_{2k+3}), \
\ldots  \in R_l.
\]
Hence, all the generators of the $T$-subspace $C(G)$ belong to
$R_l$ so $C(G) \subseteq R_l$, as claimed.

Note that $T^{(3,k+1)} \subseteq R_l$ and $T^{(3,k+1)} \not
\subseteq C(G)$ so $ C(G) \subsetneqq R_l$. By Theorem
\ref{C_limit}, $C(G)$ is a limit $T$-subspace so each $T$-subspace
$W$ such that $ C(G) \subsetneqq W$ is finitely generated. In
particular, $R_l$ is a finitely generated $T$-subspace. On the
other hand, by Proposition \ref{rk}, the $T$-subspace $R_l$ is not
finitely generated. This contradiction proves that $R_k \ne R_l$
if $k \ne l$, as required.
\end{proof}

Theorem \ref{theorem_main3}  follows  immediately from Proposition \ref{rk},
Corollary \ref{Rk_limit} and Proposition \ref{Not_equal}.

\section{Acknowledgement}
This work was partially supported by DPP/UnB and by CNPq-FAPDF PRONEX grant 2009/00091-0 (193.000.580/2009). The work of the second and the third authors was partially supported by CNPq; the work of the third author was also partially supported by FEMAT. Thanks are due to the referee whose remarks and suggestions improved the paper.

\end{document}